\documentclass{article}
\pdfoutput=1
\usepackage{nips06,times,amsmath,amsfonts,graphicx,amsthm}
\usepackage{wrapfig}
\theoremstyle{plain}
\newtheorem{mythm}{Theorem}

\newtheorem{myclaim}{Claim}

\title{Additive Non-negative Matrix Factorization for Missing Data}

\author{
Mithun Das Gupta
\\
Epson Research and Development\\
San Jose, CA.\\
\texttt{mdasgupta@erd.epson.com} \\
}

\begin{document}

\maketitle

\begin{abstract}
Non-negative matrix factorization (NMF) has previously been shown to
be a useful decomposition for multivariate data. We interpret the factorization in a new way and use it to generate missing attributes from test data. We provide a joint optimization scheme for the missing attributes as well as the NMF factors. We prove the monotonic convergence of our algorithms. We present classification results for cases with missing attributes.
\end{abstract}

\section{Introduction}
The nonnegative matrix factorization (NMF) has been shown recently to be useful for many applications in environment, pattern recognition, multimedia, text mining, and DNA gene expressions~\cite{Cooper02,Xu03,Paatero94,Li01}. NMF can be traced back to 1970s (Notes from G. Golub) and has been studied extensively by Paatero~\cite{Paatero94}. The work of Lee and Seung~\cite{Lee99,Lee00} brought much attention to NMF in machine learning and data mining fields. Various extensions and variations of NMF have been proposed recently~\cite{Ding06,Ding06a,Ding06b,Berry07,Torre06}. NMF, in its most general form, can be described by the following factorization
\begin{equation}\label{EQN:NMF_2}
   X^{d\times N} = W^{d\times r} H^{r \times N}
\end{equation}
where $d$ is the dimension of the data, $N$ is the number of data points (usually more than $d$) and $r < d$. Generally, this factorization has been compared with data decomposition techniques. In this sense $W$ is called the set of basis functions and the set $H$ is the data specific weights. It has been claimed by numerous researchers that such a decomposition has some favorable properties over other similar decompositions, such as PCA etc.

In the vast amount of literature present in this area, the parameter $r$ largely goes unnoticed. We pose the question, what are the fundamental differences in the decomposition for the three cases $r < d$, $r = d$ and $r > d$. The NMF decomposition for $r < d$ can be imagined to be an energy compaction process and as such, only basis vectors with higher energy remain in the decomposition. For the case of $r = d$, we can think $W$ as some sort of rotation in $d-$dimensions and as such the locally linear attributes of the data are preserved, as can be verified by finding the indices of the nearest neighbors of each data point in $X$ as well as $H$.

Now the remaining question is what happens for $r > d$. It is at this juncture that we want to concentrate our research and draw meaningful conclusions from experimental as well as empirical analysis. 
To develop a superficial motivation we look into the literature of sparse coding~\cite{Olshausen96,Garrigues08}. The basic idea which we borrow, from them is the fact that $r$ need not be limited by the dimensionality of the data. The similarity has been shown to be even greater if an additional sparseness constraint is introduced into the optimization framework~\cite{Hoyer04}. We motivate our analysis from a classification point of view. In the actual application domain we would like to handle missing attributes.

\begin{wrapfigure}{r}{0.5\textwidth}
  \begin{center}
    \includegraphics[width=0.5\textwidth]{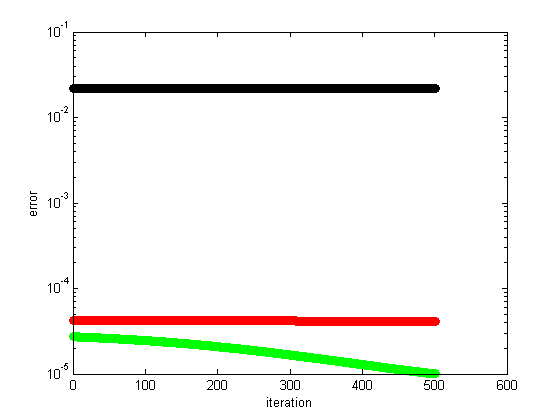}
    \caption{\small{Error norm after 50000 iterations. The plots show only last 500 error norms. Black: first decomposition (traditional NMF), red: second decomposition and green: final decomposition. The last scheme ran for only 3000 iterations for convergence ($err \le 10^{-5}$).}}\label{Fig:succ_err}
  \end{center}
\end{wrapfigure}
\section{Additive NMF}
In this section we introduce the idea of additive NMF (ANMF) which can be motivated by the following scenario. Assume that given a non-negative matrix $X$, we have run NMF algorithm for a long amount of time, but due to the inherent sub-optimal nature of the NMF algorithm, we have only converged to a local optimum. Now, we can look at the residue matrix $R_1=X-WH$, and then perform the decomposition again such that we find $R_1 = W_1H_1$. By coupling the sub-optimality conditions on the original and the second decomposition, we can claim that $\|R_1\| \ge \|R_1 - W_1H_1\|$. This leads us to the generic ANMF formulation
\begin{eqnarray}\label{EQN:SNMF}
    X &=& \sum_{i=1}^k W_iH_i \\
    s.t~~&& W_i, H_i \ge 0 ~~~ \forall i
\end{eqnarray}
This decomposition is inherently equivalent to the standard NMF for $k = 1$. Given such a formulation we can write the update equations in one of the two ways

\subsection{Multi-NMF updates}
This scheme essentially means that we employ NMF updates for each value of $i$ for the residue obtained from all the previous values, namely 1 to $i-1$. The error values for $k=3$ for the scurve data is shown in Fig.~\ref{Fig:succ_err}.

\subsection{ANMF updates}
Proceeding in a way similar to Lee and Seung~\cite{Lee00}, we can write an update scheme for the ANMF scheme. Writing the update equation for $H_j^{n+1}$, we can write
\begin{equation*}
    H_j^{n+1} = H_j^{n} + \eta [W_j^TX - W_j^T(\sum_i W_iH_i)]
\end{equation*}
where we have dropped the index $k$ for simplicity. Substituting
\begin{equation*}
    \eta = \frac{H_j^{n}}{W_j^T(\sum_i W_iH_i)}
\end{equation*}
leads to a simple multiplicative update for $H$, and an analogous scheme for $W$.
\begin{equation}\label{EQN:H_update}
    H_j^{n+1} = H_j^{n}\frac{W_j^TX}{W_j^T(\sum_i W_iH_i)},~~~~~~~~W_j^{n+1} = W_j^{n}\frac{XH_j^T}{(\sum_i W_iH_i)H_j^T}
\end{equation}
\subsubsection{Convergence sketch}
Convergence of the SNMF scheme can be proved in the same manner as done by Lee and Seung~\cite{Lee00}. The auxiliary function $G(h,h^t)$ remains exactly the same for us in form, the only difference being the first order derivative which in our case is
\begin{equation*}
    \nabla F(h_j^t) = -W_j^T(X - \sum_i W_i H_i)
\end{equation*}
The minimizer for the auxiliary function can now be shown to be exactly similar to the update rules mentioned in Eqn.~\ref{EQN:H_update}.

\section{ANMF for Missing Attributes}
The training data is used to learn $W$, with $r > d$. This can be viewed as developing an over-complete dictionary from the data. The hope is that this over-complete dictionary will encode enough information, to guess the values of missing attributes, which can be further used for classification. The similar procedure for $r < d$ has no guarantee to encode extra information, since the matrix $W$ will be rank limited by the dimension $r$ and hence removing a row from $d$ might eliminate a rank dimension. The basic idea is that since NMF results in a decomposition of \textit{feature dependent} (W) and \textit{data dependent} term (H), we can remove the particular row from W for which we do not have information, and still generate a good estimate for the data dependent term H for the data point with missing attributes. A simple multiplication with the whole W then gives the approximation for the missing attributes.

The generic data imputation based classification algorithm is as follows:
\begin{itemize}
  \item Training: Assume labeled training data without missing attributes and find the decomposition
  $X_{tr} \approx W_{tr} H_{tr} = \widehat{X}_{tr}$
  \item Now keeping the same $W_{tr}$ find the decomposition
  $X_{te}M_{te} \approx W_{tr} H_{te}M_{te}$. The mask $M_{te}$ is placed to zero out the rows of $W_{tr}$ corresponding to the missing attributes. Finally, the joint estimate for the missing attributes can be obtained from $W_{tr} H_{te}=\widehat{X}_{te}$.
  \item Learn a classifier for $\widehat{X}_{tr}$. Generate the classification results for $\widehat{X}_{te}$.
\end{itemize}

Some of the advantages of the decomposition is that the training data decomposition can be done offline once, and then the learned set of basis functions $W_{tr}$ can be used for the test data transformation. Also, the classification engine does not need to perform any additional task because we convert the data back to its original dimension.

\section{Algorithm Details}
From here on, for the rest of the development, we work on a single test point $\mathbf{x} \in \mathbb{R}^d$ and present all the analysis based on a single point. The extension to multiple points $\mathbf{X}$ is straight forward. We also assume, WLOG, that the last attribute $x_d$ is the missing attribute, and follow the notations mentioned in Eqn.~\ref{EQN:Notations} for the rest of the development.
\begin{eqnarray}\label{EQN:Notations}
  \mathbf{x} = \left[
                  \begin{array}{c}
                    \mathbf{\bar{x}} \\
                    x_d \\
                  \end{array}
                \right],~~~~~~~~~~~~~~~~  \mathbf{{W}} = \left[
                  \begin{array}{c}
                    \mathbf{\overline{W}} \\
                    \mathbf{W}_d \\
                  \end{array}
                \right]
\end{eqnarray}

The optimization scheme for the observed part, $\mathbf{\bar{x}}$, can now be written as
\begin{equation}\label{EQN:NMF_3}
   \mathbf{\bar{x}} = \mathbf{\overline{W}} \mathbf{h}
\end{equation}
The update equations can be obtained directly from the update rules of Lee and Seung~\cite{Lee00}.
Once the iterations have converged we can find the missing attribute from the projection
$x_d = \mathbf{W}_d \mathbf{h}$.

\begin{mythm} The squared error
\begin{equation*}
    (1/2)(\|\mathbf{x} - \mathbf{W} \mathbf{h}\|^2
\end{equation*}
is non-increasing under the following updates
\begin{eqnarray}\label{Eqn:NewUpdateb}
    x_d^{n+1} &=& {\mathbf{W}_d\mathbf{h}_n},~~~~~~~~~~~~~~~~ \mathbf{h}_{n+1} = \mathbf{h}_n\circ\frac{\mathbf{\overline{W}}^T\mathbf{\bar{x}}}{\mathbf{\overline{W}}^T \mathbf{\overline{W}}\mathbf{h}_n}
\end{eqnarray}
where $\mathbf{x}$ and $\mathbf{W}$ are as defined in Eqn.~\ref{EQN:Notations}.
\end{mythm}

\begin{proof}
The squared error can be written as
\begin{equation}\label{Eqn:opt_scheme1}
    \min_{x_d,\mathbf{h}}~~~ F(x_d,\mathbf{h}) = (1/2)(\|x_d - \mathbf{W}_d \mathbf{h}\|^2 +
   \|\mathbf{\bar{x}} - \mathbf{\overline{W}} \mathbf{h}\|^2)
\end{equation}
Writing the first order derivatives with respect to $x_d$ and $\mathbf{h}$ and equating them to zero we get
\begin{eqnarray}
 \label{Eqn:dfx}
  \nabla_{x_d}F(x_d,\mathbf{h}) &=& (x_d - \mathbf{W}_d \mathbf{h}) = 0\\
 \label{Eqn:dfh}
  \nabla_{\mathbf{h}}F(x_d,\mathbf{h}) &=& -\mathbf{W}_d^T(x_d - \mathbf{W}_d \mathbf{h}) - \mathbf{\overline{W}}^T\lambda (\mathbf{\bar{x}} - \mathbf{\overline{W}} \mathbf{h})  \\  \nonumber
  &=& -\mathbf{W}_d^T\nabla_{x_d}F(x_d,\mathbf{h}) -\mathbf{\overline{W}}^T (\mathbf{\bar{x}}-\mathbf{\overline{W}}\mathbf{h}) \\ \nonumber
  &=& -\mathbf{\overline{W}}^T (\mathbf{\bar{x}}-\mathbf{\overline{W}}\mathbf{h})
\end{eqnarray}
The update for $x_d$ is simply obtained from Eqn.~\ref{Eqn:dfx}. Eqn.~\ref{Eqn:dfh} suggests that the update for $\mathbf{h}$ can now be obtained by solving the reduced system
\begin{equation}\label{Eqn:opt_scheme3}
    \min_{\mathbf{h}}~~~ F(\mathbf{h}) = (1/2)(\|\mathbf{\bar{x}} - \mathbf{\overline{W}} \mathbf{h}\|^2)
\end{equation} which is the same as Eqn.~\ref{EQN:NMF_3}, for which the optimum non-negative, non-increasing update has been shown to be the same as Eqn.~\ref{Eqn:NewUpdateb}~\cite{Lee00}.
\end{proof}
A similar extension can now be applied to the SNMF scheme, which leads us to the following claim:
\begin{myclaim}
The squared error
\begin{equation*}
    (1/2)(\|\mathbf{x} - \sum_i^k \mathbf{W}_i \mathbf{h}_i\|^2
\end{equation*}
is non-increasing under the following updates
\begin{eqnarray}
    x_d^{n+1} &=& \sum_i^k{\mathbf{W}_d^i\mathbf{h}_i^n},~~~~~~~~~~~~~~~~~~~~\mathbf{h}_i^{n+1} = \mathbf{h}_j^n\circ\frac{\mathbf{\overline{W}}_j^T\mathbf{\bar{x}}}{\mathbf{\overline{W}}_j^T \sum_i^k \mathbf{\overline{W}}_i\mathbf{h}_i^n}
\end{eqnarray}
where $\mathbf{x}$ is as defined in Eqn.~\ref{EQN:Notations} and $\mathbf{W}_i$'s are defined analogous to $\mathbf{W}$.
\end{myclaim}

\section{Experiments}
First we present results on manifold data as shown in Fig.~\ref{Fig:prel_res}. 
As can be seen from the results, similar color dots, which have one axis value artificially set to zero are pulled closer to same color data points on the true manifold.
\begin{figure*}[htbp!]
\begin{center}
\includegraphics[width=6cm,height=3.0cm]{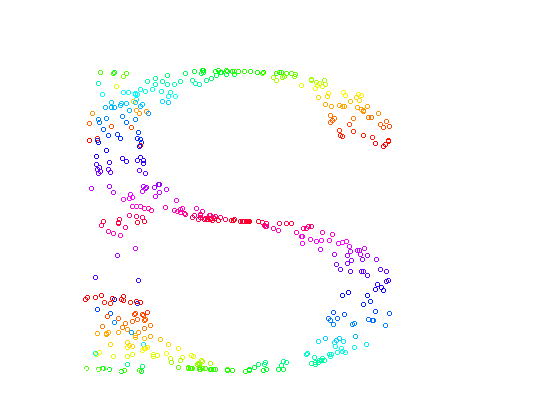}
\includegraphics[width=6cm,height=3.0cm]{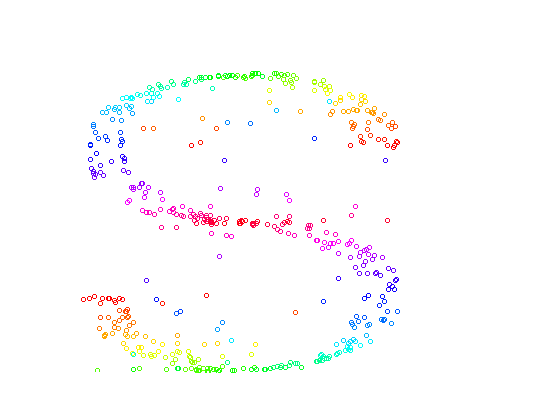}
\includegraphics[width=6cm,height=3.0cm]{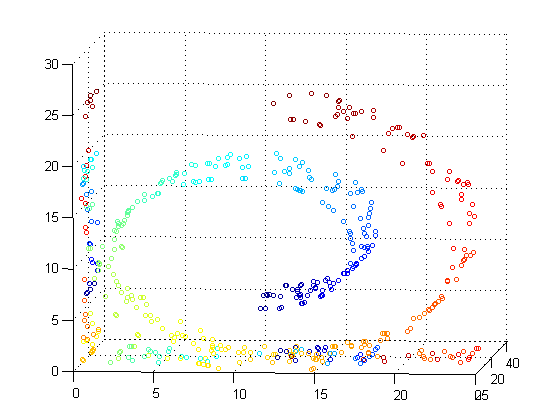}
\includegraphics[width=6cm,height=3.0cm]{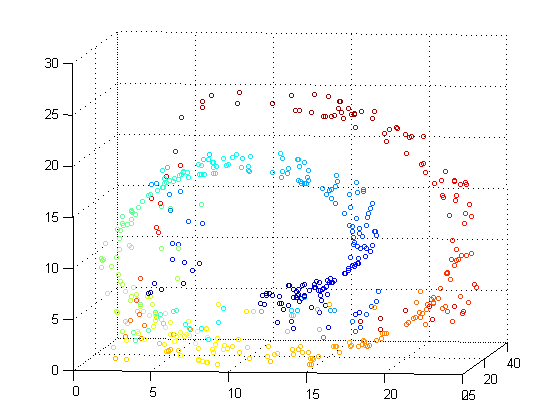}
\caption{Left: input data with missing attributes in one dimension only, right: our result.}\label{Fig:prel_res}
\end{center}
\end{figure*}

Next we present results for the WDBC1 data from UCI Machine learning repository. The data is represented as 30 dimensional vectors, with 2 possible classes. There are total 569 data points. We randomly select about 80\% of the data ar training data and the rest as testing data.

The baseline performance denotes the classification accuracy with complete data. 
We introduce the missing attributes in the following way: for each test data point we generate a 30 dimensional random vector $R \in (0,1)^{30}$. All the indices in the vector $R$ having values less than a threshold $t=0.3$ are marked for deletion. All marked indices are subsequently replaced by zeros in the test data point. This process is repeated for the entire test data set.

The comparison is shown in the following table
\begin{center}
\begin{tabular}{|c|c|c|c|}
\hline
  Dataset & Baseline & Missing 30\% (Zero substitution) & NMF with missing \\
  \hline
  \hline
  WDBC & 97 & 86.95 & 91.91 \\
  Ion & 85.91 & 73.23 & 76.05 \\
  Pima & 76.67 & 69.48 & 70.12 \\
  Echo & 88.89 & 77.78 & 88.89 \\
  \hline
\end{tabular}
\end{center}

In the next experiment we guess the value of the missing attributes, in one of the following manner: zero substitute, mean substitute, and random substitute. The results are shown in the following table. All the results are for the WDBC dataset (base accuracy $93.07\%$).
\begin{center}
\begin{tabular}{|c|c|c|c|c|}
\hline
  \% missing & Zero & Mean & Random & NMF \\
  \hline
  \hline
  10 & 92.17 & 91.30 & 92.17 & 95.17 \\
  20 & 89.57 & 85.22 & 89.57 & 93.30 \\
  30 & 81.74 & 68.70 & 81.74 & 91.30 \\
  40 & 80.35 & 64.16 & 80.35 & 87.61 \\
  \hline
\end{tabular}
\end{center}

{\small{
\bibliographystyle{plain}
\bibliography{nmf_md}
}
}

\end{document}